\documentclass{article}
\pdfoutput=1

\usepackage{amsmath}
\usepackage{amssymb}
\usepackage{amsthm}
\usepackage{authblk}
\usepackage[square,numbers]{natbib}
\usepackage{xspace}
\usepackage[all]{xy}
\usepackage{subfigure}
\usepackage{tikz}
\usetikzlibrary{arrows,decorations.pathmorphing,backgrounds,positioning,fit,through}

\newtheorem{theorem}{Theorem}
\newtheorem{cor}[theorem]{Corollary}
\newtheorem{obs}{Observation}
\newtheorem{lemma}{Lemma}

\newtheorem{dnt}{Definition}

\newtheorem{problem}{Problem}

\newcommand{\ag}{\textsc{Autograph}\xspace}
\newcommand{\sig}{\textsc{Signature}\xspace}

\title{We Found the Smallest Non-Autograph}
\author[1]{Ben Baumer\thanks{Clark Science Center, 44 College Lane, Northampton, MA 01063, \texttt{bbaumer@smith.edu}}}
\author[1]{Yijin Wei\thanks{\texttt{ywei@smith.edu}}}
\author[2]{Gary Bloom}
\affil[1]{Department of Mathematics and Statistics, Smith College}
\affil[2]{Department of Computer Science, City College}

\begin{document}
\maketitle

\begin{abstract}
Suppose that $G$ is a simple, vertex-labeled graph and that $S$ is a multiset. Then if there exists a one-to-one mapping between the elements of $S$ and the vertices of $G$, such that edges in $G$ exist if and only if the absolute difference of the corresponding vertex labels exist in $S$, then $G$ is an \emph{autograph}, and $S$ is a \emph{signature} for $G$. While it is known that many common families are graphs are autographs, and that infinitely many graphs are not autographs, a non-autograph has never been exhibited. In this paper, we identify the smallest non-autograph: a graph with 6 vertices and 11 edges. Furthermore, we demonstrate that the infinite family of graphs on $n$ vertices consisting of the complement of two non-intersecting cycles contains only non-autographs for $n \geq 8$. 
\\\\
\textbf{graph labeling, difference graphs, autographs, monographs}\\
\textbf{MSC Primary 05C78; Secondary 05C60;}
\end{abstract}

\section{Introduction}

Let $G = (V,E)$ be a simple, vertex-labeled graph, and let $S$ be a multiset. If there exists a bijection $\pi : S \rightarrow V$ such that for every $s,t \in S$, the edge $(\pi(s), \pi(t)) \in E$ if and only if $|s-t| \in S$, then we say that $G$ is an \emph{autograph} and that $S$ is a \emph{signature} for $G$. For example, consider the path graph on three vertices $P_3$ shown in Figure \ref{fig:path}. This is an autograph, since one possible signature for this graph is $S = \{1,2,4\}$.

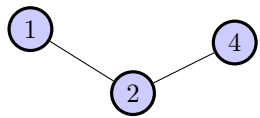
\begin{figure}[h]
	\centering
	\tikzset{
	node/.style={circle,inner sep=1mm,minimum size=0.2cm,draw,very thick,black,fill=blue!20,text=black},
	nondirectional/.style={thin,black},
}

\begin{tikzpicture}[scale=5]
	\node [node] (v0) at (-0.375458, -0.469825)	{1};
	\node [node] (v1) at (-0.102779, -0.640131)	{2};
	\node [node] (v2) at (0.167645, -0.505400)	{4};
	\path [nondirectional] (v0) edge (v1);
	\path [nondirectional] (v1) edge (v2);
\end{tikzpicture}
	\caption{The path graph $P_3$ on three vertices. This is an autograph realized by the signature $\{1,2,4\}$. The edge between the vertices labeled 1 and 2 exists because $|1-2| = 1 \in S$, and similarly for the vertices labeled 2 and 4. However, the edge between 1 and 4 does not exist because $|1-4| = 3 \notin S$. }
	\label{fig:path}
\end{figure}

The use of a signature to define and categorize graphs is useful because it provides an exceptionally compact data structure for graphs -- an autograph can be represented using only $n$ numbers, whereas a non-sparse graph requires $O(n^2)$. The study of autographs -- or undirected \emph{difference graphs} -- has led to the knowledge that many families of graphs are autographs, and indeed, signatures for a large number of common graph families have been discovered. For example, the set $S = \{1, 2, \ldots, n\}$ provides a signature for the complete graph $K_n$ on $n$ vertices. However, it is also known that infinitely many graphs on $n$ vertices are not autographs, and further, it is conjectured that nearly all graphs are not autographs. Yet while \cite{bloom1984acg} proved that infinitely-many non-autographs exist, none has ever been demonstrated. In this paper, we prove that the graph $M_6$ on 6 vertices shown in Figure \ref{fig:m6} is the smallest non-autograph. Furthermore, the infinite family of graphs $G_n = K_n \setminus 2 C_{n/2}$, $n$ even, consisting of a complete graph with two non-intersecting cycles deleted, contains only non-autographs for $n \geq 8$. An illustration of $G_{10}$ is shown in Figure \ref{fig:G}.

\begin{figure}[ht]
	\centering
	\subfigure[$M_6$]{
		\tikzset{
	node/.style={circle,inner sep=1mm,minimum size=0.2cm,draw,very thick,black,fill=blue!20,text=black},
	nondirectional/.style={thin,black},
}

\begin{tikzpicture}[yscale=1]
	\node [node] (v0) at (0, 1)	{?};
	\node [node] (v1) at (1, 0.5)	{?};
	\node [node] (v2) at (1, -0.5)	{?};
	\node [node] (v3) at (0, -1)	{?};
	\node [node] (v4) at (-1, -0.5)	{?};
	\node [node] (v5) at (-1, 0.5)	{?};
	\path [nondirectional] (v0) edge (v1);
	\path [nondirectional] (v0) edge (v5);
	\path [nondirectional] (v0) edge (v3);
	\path [nondirectional] (v3) edge (v4);
	\path [nondirectional] (v3) edge (v2);
	\path [nondirectional] (v1) edge (v2);
	\path [nondirectional] (v1) edge (v4);
	\path [nondirectional] (v1) edge (v5);
	\path [nondirectional] (v2) edge (v4);
	\path [nondirectional] (v2) edge (v5);
	\path [nondirectional] (v4) edge (v5);
\end{tikzpicture}
	} 
	\qquad		
	\subfigure[$\overline{M_6}$]{
		\tikzset{
	node/.style={circle,inner sep=1mm,minimum size=0.2cm,draw,very thick,black,fill=blue!20,text=black},
	nondirectional/.style={thin,black},
}

\begin{tikzpicture}[yscale=1]
	\node [node] (v0) at (0, 1)	{?};
	\node [node] (v1) at (1, 0.5)	{?};
	\node [node] (v2) at (1, -0.5)	{?};
	\node [node] (v3) at (0, -1)	{?};
	\node [node] (v4) at (-1, -0.5)	{?};
	\node [node] (v5) at (-1, 0.5)	{?};
	\path [nondirectional] (v0) edge (v2);
	\path [nondirectional] (v0) edge (v4);
	\path [nondirectional] (v3) edge (v1);
	\path [nondirectional] (v3) edge (v5);
\end{tikzpicture}
	}
	\caption{The smallest non-autograph, $M_6$ (left), and its complement (right).}
	\label{fig:m6}
\end{figure}
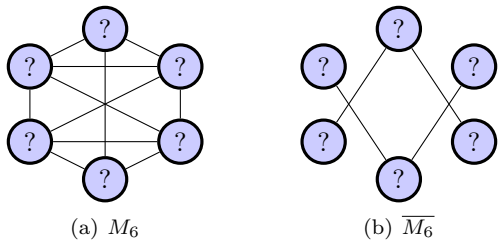

\begin{figure}
		\centering
		\subfigure[$G_{10} = \overline{2C_5}$]{
			\begin{xy}
	(20, 0)*+{\bullet}="a";
	(16.18, 11.76)*+{\bullet}="b";
	(6.18, 19.02)*+{\bullet}="c";
	(-6.18, 19.02)*+{\bullet}="d";
	(-16.18, 11.76)*+{\bullet}="e";
	(-20, 0)*+{\bullet}="f";
	(-16.18, -11.76)*+{\bullet}="g";
	(-6.18, -19.02)*+{\bullet}="h";
	(6.18, -19.02)*+{\bullet}="i";
	(16.18, -11.76)*+{\bullet}="j";
		{\ar@{-} "a"; "b"};
		{\ar@{-} "a"; "d"};
		{\ar@{-} "a"; "e"};
		{\ar@{-} "a"; "f"};
		{\ar@{-} "a"; "g"};
		{\ar@{-} "a"; "h"};
		{\ar@{-} "a"; "j"};
		{\ar@{-} "b"; "c"};
		{\ar@{-} "b"; "e"};
		{\ar@{-} "b"; "f"};
		{\ar@{-} "b"; "g"};
		{\ar@{-} "b"; "h"};
		{\ar@{-} "b"; "i"};
		{\ar@{-} "c"; "d"};
		{\ar@{-} "c"; "f"};
		{\ar@{-} "c"; "g"};
		{\ar@{-} "c"; "h"};
		{\ar@{-} "c"; "i"};
		{\ar@{-} "c"; "j"};
		{\ar@{-} "d"; "e"};
		{\ar@{-} "d"; "g"};
		{\ar@{-} "d"; "h"};
		{\ar@{-} "d"; "i"};
		{\ar@{-} "d"; "j"};
		{\ar@{-} "e"; "f"};
		{\ar@{-} "e"; "h"};
		{\ar@{-} "e"; "i"};
		{\ar@{-} "e"; "j"};
		{\ar@{-} "f"; "g"};
		{\ar@{-} "f"; "i"};
		{\ar@{-} "f"; "j"};
		{\ar@{-} "g"; "h"};
		{\ar@{-} "g"; "j"};
		{\ar@{-} "h"; "i"};
		{\ar@{-} "i"; "j"};
	\end{xy}
		} 
		\qquad		
		\subfigure[$2C_5$]{
			\begin{xy}
		(20, 0)*+{\bullet}="a";
	(16.18, 11.76)*+{\bullet}="b";
	(6.18, 19.02)*+{\bullet}="c";
	(-6.18, 19.02)*+{\bullet}="d";
	(-16.18, 11.76)*+{\bullet}="e";
	(-20, 0)*+{\bullet}="f";
	(-16.18, -11.76)*+{\bullet}="g";
	(-6.18, -19.02)*+{\bullet}="h";
	(6.18, -19.02)*+{\bullet}="i";
	(16.18, -11.76)*+{\bullet}="j";
		{\ar@{-} "a"; "c"};
		{\ar@{-} "c"; "e"};
		{\ar@{-} "e"; "g"};
		{\ar@{-} "g"; "i"};
		{\ar@{-} "i"; "a"};
		{\ar@{-} "b"; "d"};
		{\ar@{-} "d"; "f"};
		{\ar@{-} "f"; "h"};
		{\ar@{-} "h"; "j"};
		{\ar@{-} "j"; "b"};
	\end{xy}
	}
	\caption{Illustration of $G_{10} = \overline{2C_5}$ and $2C_5$. We show that this graph is not an autograph.}
	\label{fig:G}
\end{figure}
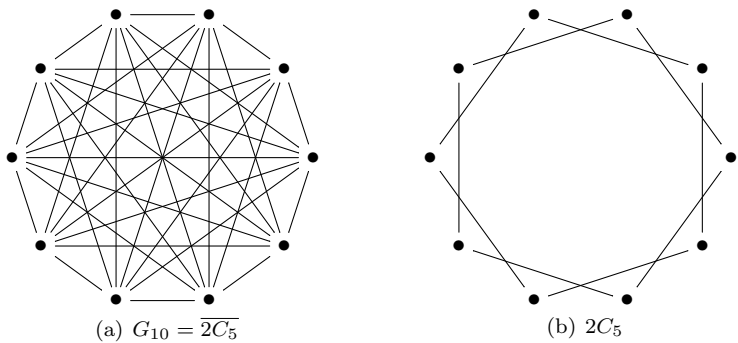

\subsection{Motivation}

The bijective nature of autographs becomes useful in a variety of practical contexts. The primary advantage is that since the adjacency matrix for an autograph can be computed from its signature, an autograph can be stored using $O(n)$ space, compared to $O(n^2)$ for a non-sparse graph. For large graphs, this order of magnitude savings can be of considerable practical value. For example, these savings could improve reliability if such information had to be transmitted over a noisy channel. Given its brevity, the representation of an autograph by its signature provides some inherent increase in the probability of a successful transmission over a network.

Two problems are of interest:

\begin{problem}
[\ag] Given a simple graph $G$, determine if $G$ is an autograph.
\end{problem}

\begin{problem}
[\sig] Given an autograph $G$, find a signature $S$ for $G$.
\end{problem}

In Section~\ref{sec:conclusion}, we discuss some computational complexity concerns of these two problems. 

\subsection{Related Work}

Within graph labeling, there is substantial interest in \emph{graceful labelings}, dating back to the work of~\cite{golomb1972ng} and~\cite{rosa1966certain}. A graph $G=(V,E)$ is considered graceful if the vertices can be numbered with integers chosen from $[0,|E|]$ such that every edge receives a distinct integer label defined by the absolute difference of the labels of its neighboring vertices. This theory led, in particular, to the Ringel-K\"{o}tzig graceful tree conjecture, which poses the question of whether all trees are graceful~\citep{bloom1979crk,rosa1966certain}.

However, while graphs may or may not admit a graceful labeling, the process of translating between vertex labels and graphs is not bijective. That is, if given a graph, one might be able to find a graceful labeling, but if given a graceful vertex labeling, one cannot recover the original graph. The desire to make this process bijective led to the notion of an \emph{autograph}, as defined above. \cite{bloom1979can} found all signatures for graphs with at most four vertices, as well as trees, paths, cycles, complete graphs, pyramids and $n$-prisms. \cite{gervacio1983wheels} investigated wheel graphs, and proved that they are \emph{proper}\footnote{A proper autograph contains only positive signature values.} autographs for $n=3, 4$ and $6$ only. \cite{harary1990sum} used the \emph{difference graph} terminology, wherein the edge labels may come from the difference (not necessarily absolute) of the vertex labels. In this realm \cite{sonntag2004difference} found conditions for building directed graphs from smaller difference graphs and also found that all cacti with girth at least 6 are autographs~\citep{sonntag2004cacti}. 

If the signature for a graph contains distinct elements, then the graph is known as a \emph{monograph}. \cite{ук2007several} studied the properties of monographs and discovered signatures for cycles, fan graphs, kite graphs and necklaces. \cite{seoud2011difference} listed signatures for all graphs of order 5 and discovered signatures for gear graphs, triangular snakes, and dragons, among other things. The related construction of \emph{mod difference digraphs}, in which $S = [n] = \{1,2,\ldots,n\}$ and the edge labels are taken modulo $n$, were explored by \cite{hedge2009on}.

\subsection{Our Contributions}

While most of the work cited above has focused on discovering signatures for families of graphs, thereby proving that they are autographs, comparatively less work has investigated the properties of \emph{non-autographs}. \cite{bloom1984acg} proved that there are infinitely many graphs on $n$ vertices that are not autographs, but did not produce one. In that effort, monographs of low codegree -- i.e. graphs in which each vertex is connected to most other vertices -- were used to put constraints on possible signature values for monographs. We extend this work to demonstrate that a certain family of graphs of low codegree are not autographs. Moreover, while~\citep{bloom1979can} found proper autograph signatures for all simple non-isomorphic graphs with at most 5 vertices, they exhibited three graphs with 6 vertices that were not proper autographs. We show that one of these graphs (Fig. \ref{fig:m6}) is not an autograph, making it the smallest non-autograph. 

\begin{theorem}
$M_6$ is the smallest non-autograph. 
\end{theorem}

\begin{theorem}
Let $G_n = K_n \setminus 2 C_{n/2}$ be the complete graph on $n$ vertices ($n$ is even), with two non-intersecting cycles deleted. Then the infinite family of graphs $G_n$ contains only non-autographs for $n \geq 8$. 
\end{theorem}

This paper proceeds as follows: in Section \ref{sec:prelim} we develop some preliminary notions and lemmas that constrain possible signatures for graphs based on certain properties. In particular, we investigate the relationship between autographs and signatures composed of arithmetic progressions. In Section \ref{sec:m6} we prove that $M_6$ is the smallest non-autograph. In Section \ref{sec:g_n} we prove that $G_n$ is a non-autograph for $n \geq 8$. We conclude with open problems and thoughts on future work. 

\section{Preliminaries}
\label{sec:prelim}

Recall that if the signature $S(G)$ for a graph $G$ is a \emph{set} (i.e. not a multiset), then that graph is a \emph{monograph}. We show presently that both $M_6$ and $G_n$ are monographs.

\begin{obs}\label{obs:deg_of_0}
If $0 \in S(G)$, then any vertex $v \in V(G)$ with $\pi^{-1}(v) = 0$ is adjacent to every vertex with a non-negative signature value. 
\end{obs}

\begin{proof}
For any other vertex $w \in V(G)$ with $\pi^{-1}(w) > 0$, an edge between $v$ and $w$ exists, since $|\pi^{-1}(w) - \pi^{-1}(v)| = |\pi^{-1}(w)| = \pi^{-1}(w) \in S(G)$. 
\end{proof}

\begin{obs}\label{diff_neigobor}
If two vertices $v, w \in V(G)$ have the same label, then $v$ and $w$ have the same neighborhood. 
\end{obs}

\begin{proof}
Let $N_G(v), N_G(w)$ denote the set of vertices in $G$ adjacent to vertices $v,w \in V(G)$, respectively, where $\pi^{-1}(v) = \pi^{-1}(w)$.  
For any $x \in N_G(v)$, $|\pi^{-1}(x)-\pi^{-1}(w)| = |\pi^{-1}(x)-\pi^{-1}(v)| \in S(G)$, since $v, w$ have the same label. Thus, $x \in N_G(w)$. Similarly, for any $y \in N(w)$, $|\pi^{-1}(y)- \pi^{-1}(w)| = |\pi^{-1}(y)- \pi^{-1}(v)| \in S(G)$ and $y \in N_G(v)$. Therefore, $N_G(v) = N_G(w)$.
\end{proof}

It follows that if every vertex in a graph $G$ has a different set of neighbors, then $G$ is a monograph. Both $M_6$ and $G_n$ have this property, and accordingly we restrict our attention to monographs throughout this paper.

\subsection{Notation}

We assume that $G$ is a monograph with signature $S(G) = S = \{s_1, \cdots, s_n\}$ where $s_1 < s_2 < \ldots < s_n$. 
Let $m$ be the number of negative signature values in $S(G)$, and set $S^- = \{s_1,\cdots,s_m\}$ and $S^+ = \{s_{m+1},\cdots, s_n\}$. Thus, we have that
$$
	S(G) = S^- \cup S^+, \qquad \text{ with } \underbrace{s_1 < s_2 < \ldots s_m}_{negative} < \underbrace{s_{m+1} < \ldots < s_{n-1} <  s_n}_{non-negatitive} 
$$ 
If $\delta$ is the minimum degree of any vertex in $G$, then $k = n-1-\delta$ is the \emph{codegree} of $G$ -- i.e. the maximum number of vertices to which each vertex in $G$ is not adjacent. 

Let $S= \{ ts, (t+1)s, ..., (n-1+t)s\}$ be a signature for a proper monograph consisting of an arithmetic progression of $n$ terms, beginning at $ts$ and continuing in integer steps of size $s$. Note that $S$ defines a monograph. We adopt the simpler notation $[t]_s^n$ for such a signature.

\begin{dnt}
For any signature value $s \in S(G)$, where $G$ is a monograph, the difference set $D(s) = \{ |s - t| : t \in S, t \neq s\}$ is the set of all possible edge labels associated with the vertex $\pi(s)$ in $G$.
\end{dnt}

\subsection{Basic Results}

\begin{lemma}\label{lemma:kth-largest}
In a monograph, the vertex labeled with the $r^{th}$ largest non-negative signature value can be adjacent to at most $r-1$ vertices labeled with negative signature values. 
\end{lemma}


\begin{proof}
Note that $s_{n-r+1}$ is the $r^{th}$ largest positive signature value, and $r-1$ vertices have signature values greater than  $s_{n-r+1}$. Assume by contradiction that $s_{n-r+1}$ has at least $r$ neighbors with negative signature values $s_i$, for $1 \leq i \leq m$. For all such $s_i$, $|s_{n-r+1} - s_i| = s_{n-r+1} + |s_i| \in S^+$. Since $G$ is a monograph, we have identified at least $r$ distinct values in $S^+$ that are greater than $s_{n-r+1}$, a contradiction.
\end{proof}

\begin{cor}\label{cor:neg_sig_values}
The number of negative signature values in any monograph with $n$ vertices is at most $ n - 1 - deg(s_n)$.
\end{cor}

\begin{proof}
Apply Lemma \ref{lemma:kth-largest} to the maximum signature value $s_n$, and we get that $s_n$ can only be adjacent to vertices with positive signature values, of which there are $n-m-1$ other than $s_n$. Thus, 
$$
	\deg(s_n) \leq n-m-1 \Rightarrow m \leq n - 1 - \deg(s_n)
$$
\end{proof}

\begin{lemma} (The Maximum Element Lemma for Monographs) \label{lemma:max}
In a monograph, if 0 is not in the signature and $s_n$ is adjacent to all vertices with positive signature values, then
$$
	s_n = s_{m+1} + s_{n-1} = s_{m+2} + s_{n-2} = \cdots = s_{\lfloor \frac{m+n}{2} \rfloor} + s_{\lceil \frac{m+n}{2} \rceil}
$$

\end{lemma}

\begin{proof}
Consider the difference set $D(s_n)$ of $s_n$. In decreasing order,
 
$$D(s_n)=\{s_n-s_{m+1}, s_n-s_{m+2},...,s_n-s_{n-1}\}$$
  
By Lemma~\ref{lemma:kth-largest}, $s_n$ is only adjacent to vertices with non-negative signature values. We also know $0 \notin S$. Thus, the $n-m-1$ terms of $D(s_n)$ are distinct and contained in $S^+\setminus\{s_n\}$. Note that $S^+\setminus\{s_n\}$ has $n-m-1$ positive signature values. Thus, there is a one-to-one correspondence between $D(s_n)$ and $S^+\setminus\{s_n\}$, that is $D(s_n)=S^+\setminus\{s_n\}$. In decreasing order, 
  
$$S^+\setminus\{s_n\}=\{s_{n-1},s_{n-2},...,s_{m+1}\}$$

Since $D(s_n)=S^+\setminus\{s_n\}$,
$$s_n-s_{m+1}=s_{n-1},\quad s_n-s_{m+2}=s_{n-2} \ ...$$

The result now follows. 
\end{proof}

\subsection{Arithmetic Progressions}
\label{sec:arith}

Recall that $S= [t]_s^n = \{ ts, (t+1)s, ..., (n-1+t)s\}$ is a signature for a proper monograph consisting of an arithmetic progression of $n$ terms, beginning at $ts$ and continuing in integer steps of size $s$. This section concerns the properties of graphs resulting from deleting members of $S=[1]_s^n$.

\begin{lemma} \label{lemma:one_deletion_degrees}
Let $S = [1]_s^n = \{s, 2s, ..., ns\}$ be a signature for $K_n$. If the vertex of signature value $is$, for $1\leq i \leq n$, is removed from $K_n$, the degree of $ps$, where $p\neq i$, $1\leq p\leq n$ is
$$\deg(ps) = \epsilon_i+\begin{cases} n-4 & \text{if }  p \in(i,n-i] \\ n-2 & \text{if } p\in (n-i,i) \\ n-3 &\text{otherwise} \end{cases}$$
where $$
\epsilon_i = \begin{cases} 1 & \text{if } p= 2i,\  i\leq \lfloor\frac{n}{2}\rfloor \\ 0 & \text{otherwise} \end{cases}$$
\end{lemma}

\begin{figure}
	\centering
	\begin{tikzpicture}[domain=0:1,xscale=1, yscale=1
,sensor/.style={circle,draw=black,fill=red!50,thick,opacity=0.75,inner sep=0pt,minimum size=2mm}
,dummy/.style={circle,draw=black,fill=red!50,thick,opacity=0.75,inner sep=0pt,minimum size=1mm}
,moved/.style={circle,draw=black,fill=blue!50,thick,opacity=0.75,inner sep=0pt,minimum size=2mm}]
\draw[-] (1,0) -- (10,0) coordinate (x axis);
	\foreach \x/\xtext in {1, 2, ..., 10}
		\draw (\x,1pt) -- (\x,-1pt) node[anchor=north] {$\xtext$};
\node [moved] (i) at (3, 0) [label=above:$is$] {};
\node [dummy] (l) at (5, 0) [label=above:$(p-i)s$] {};
\node [sensor] (p) at (8, 0) [label=above:$ps$] {};
\node [dummy] (r) at (11, 0) [label=above:$(p+i)s$] {};
\path [bend left,dashed] (p) edge (i);
\path [bend left,dashed] (l) edge (p);
\path [bend left,dotted] (p) edge (r);
\end{tikzpicture}
	\caption{Example of one deletion from a complete graph signature. Illustrated is the case where $n=10, i=3$, and $p=8$. For the vertex $ps$, if $is$ is deleted, then the two dashed edges are lost, but the one dotted edge never existed. Thus, the degree of $ps$ in the resulting graph is $n-3$. This is case 3 in the proof below ($p > \max(n-i,i)$.)}
	\label{fig:delete1}
\end{figure}
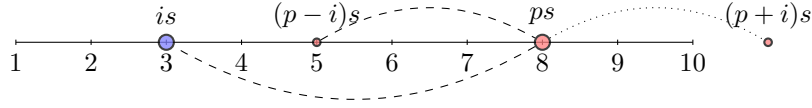

\begin{proof}
We will first consider the general case $p\neq 2i$. Before the deletion of $is$, $\deg(ps)=n-1$ for all $1\leq p\leq n$. The vertices whose degrees are potentially affected by the deletion of $is$ have signature values $(p-i)s$ and $(p+i)s$. Note that $(p-i)s<ns$ and $(p+i)s>0$ always hold. Thus, we will discuss the relationship of $(p-i)s$ with 0, and $(p+i)s$ with $ns$ to determine whether they are in the signature or affected by the deletion. We illustrate one case in Figure \ref{fig:delete1}, but the reader is encouraged to envision how the other cases could be similarly illustrated. 

\begin{enumerate}
\item $p+i\leq n$ and $p-i>0 \Rightarrow i<p\leq n-i$. \\
Both $(p+i)s$ and $(p-i)s$ are in the signature. Thus, three edges adjacent to $ps$: $(is,ps),((p-i)s,ps),((p+i)s,ps)$ are deleted after the deletion of the vertex $is$ and $\deg(ps)=n-1-3=n-4$.

\item $p+i\leq n$ and $p-i< 0 \Rightarrow p\leq\min(n-i,i)$. \\
$(p+i)s$ is in the signature and $(p-i)s$ is not. Thus, two edges adjacent to $ps$: $(is,ps),((p+i)s,ps)$ are deleted and $\deg(ps)=n-1-2=n-3$.

\item $p+i> n$ and $p-i> 0 \Rightarrow p>\max(n-i,i)$. \\
$(p-i)s$ is in the signature and $(p+i)s$ is not. Two edges are deleted and $\deg(ps)=n-1-2=n-3$. An example of this case is illustrated in Figure \ref{fig:delete1}. 

\item $p+i> n$ and $p-i< 0 \Rightarrow n-i<p< i$. \\
Neither $(p+i)s$ nor $(p-i)s$ is in the signature. Only the edge $(is, ps)$ is deleted and $\deg(ps)=n-1-1=n-2$. 
\end{enumerate}
If $p=2i$, then $(p-i)s$ and $is$ are the same vertex. That is, $(ps,(p-i)s)$ and $(ps,is)$ are the same edge. Therefore, we overcounted the number of deleted edges by 1 and this defines $\epsilon_i$.
\end{proof}

This result can be extended to the removal of two signature values, and we will use it in the proof of the main result. 

\begin{lemma} \label{lemma:two_deletions_degree}
Let $S = [1]_s^n = \{s, 2s, ..., ns\}$ be a signature for $K_n$. If vertices of signature value $is$ and $js$, for $1\leq i<j\leq n$, are removed from $K_n$, the degree of $ps$, where $p\neq i,j$, $1\leq p\leq n$ is
\nonumber
\begin{multline}
\deg(ps) = \epsilon_i+\epsilon_j+\gamma_{ij}+\beta_{ij}\\
+\begin{cases} n-7 & \text{if }  j<p \leq n-j \\ n-6 & \text{if } \max(j,n-j)<p\leq n-i \text{ or } i<p\leq \min(n-j,j)\\ n-4 & \text{if } \max(i,n-i)<p<j \text{ or } n-j<p\leq\min(n-i,i)\\ n-3 & \text{if }  n-i<p <i \\ n-5 &\text{otherwise} \end{cases}
\end{multline}
where $$
\gamma_{ij} = \begin{cases} 2 & \text{if } p= i+j \\ 0 & \text{otherwise} \end{cases} , \qquad \beta_{ij} = \begin{cases} 1 & \text{if } p= j-i \\ 0 & \text{otherwise} \end{cases}$$ 
and $\epsilon_i$ and $\epsilon_j$ are defined as before. 
\end{lemma}

\begin{figure}
	\centering
	\begin{tikzpicture}[domain=0:1,xscale=1, yscale=1
,main/.style={circle,draw=black,fill=red!50,thick,opacity=0.75,inner sep=0pt,minimum size=2mm}
,dummy/.style={circle,draw=black,fill=red!50,thick,opacity=0.75,inner sep=0pt,minimum size=1mm}
,deleted/.style={circle,draw=black,fill=blue!50,thick,opacity=0.75,inner sep=0pt,minimum size=2mm}]
\draw[-] (1,0) -- (12,0) coordinate (x axis);
	\foreach \x/\xtext in {1, 2, ..., 12}
		\draw (\x,1pt) -- (\x,-1pt) node[anchor=north] {$\xtext$};
\node [deleted] (i) at (2, 0) [label=above:$is$] {};
\node [deleted] (j) at (3, 0) [label=above:$js$] {};
\node [dummy] (li) at (3, 0) [] {};
\node [dummy] (ri) at (7, 0) [] {};
\node [main] (p) at (5, 0) [label=above:$ps$] {};
\node [dummy] (lj) at (2, 0) [] {};
\node [dummy] (rj) at (8, 0) [label=above:$(p+j)s$] {};
\path [bend left,dashed] (p) edge (i);
\path [bend left,dashed] (p) edge (li);
\path [bend left,dashed] (p) edge (ri);
\path [bend left,dashed] (j) edge (p);
\path [bend left,dashed] (lj) edge (p);
\path [bend left,dashed] (rj) edge (p);
\end{tikzpicture}
	\caption{Example of two deletions from a complete graph signature. Illustrated is the case where $n=12, i=2, j=3$, and $p=5$. For the vertex $ps$, if $is$ and $js$ are deleted, then the six dashed edges are lost, but we have double-counted two of them. Thus, the degree of $ps$ in the resulting graph is $n-5$. This is case 1 in the proof below, with $\gamma_{23} = 2$. }
	\label{fig:delete2_again}
\end{figure}
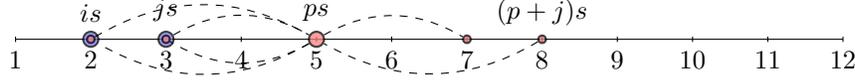

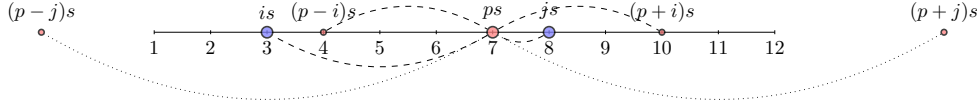
\begin{figure}
	\centering
	\scalebox{0.75}{\begin{tikzpicture}[domain=0:1,xscale=1, yscale=1
,main/.style={circle,draw=black,fill=red!50,thick,opacity=0.75,inner sep=0pt,minimum size=2mm}
,dummy/.style={circle,draw=black,fill=red!50,thick,opacity=0.75,inner sep=0pt,minimum size=1mm}
,deleted/.style={circle,draw=black,fill=blue!50,thick,opacity=0.75,inner sep=0pt,minimum size=2mm}]
\draw[-] (1,0) -- (12,0) coordinate (x axis);
	\foreach \x/\xtext in {1, 2, ..., 12}
		\draw (\x,1pt) -- (\x,-1pt) node[anchor=north] {$\xtext$};
\node [deleted] (i) at (3, 0) [label=above:$is$] {};
\node [deleted] (j) at (8, 0) [label=above:$js$] {};
\node [dummy] (li) at (4, 0) [label=above:$(p-i)s$] {};
\node [dummy] (ri) at (10, 0) [label=above:$(p+i)s$] {};
\node [main] (p) at (7, 0) [label=above:$ps$] {};
\node [dummy] (lj) at (-1, 0) [label=above:$(p-j)s$] {};
\node [dummy] (rj) at (15, 0) [label=above:$(p+j)s$] {};
\path [bend left,dashed] (p) edge (i);
\path [bend left,dashed] (li) edge (p);
\path [bend left,dashed] (p) edge (ri);
\path [bend left,dashed] (j) edge (p);
\path [bend left,dotted] (p) edge (lj);
\path [bend left,dotted] (rj) edge (p);
\end{tikzpicture}}
	\caption{Example of two deletions from a complete graph signature. Illustrated is the case where $n=12, i=3, j=8$, and $p=7$. If $is$ and $js$ are deleted, then the four dashed edges are lost, but the two dotted edges never existed. Thus, the degree of $ps$ in the resulting graph is $n-5$. This is case 9 in the proof below. }
	\label{fig:delete2}
\end{figure}

\begin{proof}
We will first consider the general case where $p \notin \{2i, 2j, j-i, j+i\}$. The vertices whose degree are potentially affected by the deletion of $is, js$ have signature values of $(p-j)s, (p+j)s, (p-i)s$ and $(p+i)s$. Note that $(p-i)s<ns, (p-j)s<ns, (p+j)s>0$ and $(p+i)s>0$ always hold. Thus, we will discuss the relationship of $(p-i)s$ and $(p-j)s$ with 0, $(p+i)s$ and $(p+j)s$ with $ns$ to determine whether these four values are in the signature or affected by the deletion. See Figures~\ref{fig:delete2_again} and~\ref{fig:delete2} for examples. 
\begin{enumerate}
\item $p-j>0$ and $p+j\leq n\Rightarrow j<p\leq n-j$\\
Since $j>i$, we know $p-i>0$ and $p+i\leq n$. Then $(p+j)s, (p-j)s, (p+i)s$ and $(p-i)s$ are all in the signature. Thus, 6 edges adjacent to $ps$: $(is,ps),((p-i)s,ps),((p+i)s,ps), (js,ps),((p-j)s,ps),((p+j)s,ps)$ are deleted and $\deg(ps)=n-1-6=n-7$. An example is illustrated in Figure~\ref{fig:delete2_again}.

\item $p-j>0,\ p+j>n$ and $p+i\leq n\Rightarrow \max(j,\ n-j)<p\leq n-i$ \\
Since $j>i$ and $p-j>0$, we know $p-i>0$. Then $(p-j)s, \ (p+i)s$ and $(p-i)s$ are in the signature. Thus, 5 edges adjacent to $ps$: $(is,ps),((p-i)s,ps),((p+i)s,ps), (js,ps),((p-j)s,ps)$ are deleted and $\deg(ps)=n-1-5=n-6$. 

\item $p-j>0,\ p+j>n$ and $p+i>n\Rightarrow p>\max(j,\ n-i)$ \\
Since $j>i$ and $p-j>0$, we know $p-i>0$. Then $(p-j)s$ and $(p-i)s$ are in the signature. Thus, 4 edges adjacent to $ps$: $(is,ps),((p-i)s,ps), (js,ps),((p-j)s,ps)$ are deleted and $\deg(ps)=n-1-4=n-5$. 

\item $p+j\leq n,\ p-j\leq 0$ and $p-i> 0\Rightarrow i<p\leq \min(j,\ n-j)$ \\
Since $p+j\leq n$ and $j>i$, we know $p+i\leq n$. Then $(p-i)s, \ (p+j)s$ and $(p+i)s$ are in the signature. Thus, 5 edges adjacent to $ps$: $(is,ps),((p-i)s,ps),(js,ps), \ ((p+j)s,ps), \ ((p-j)s,ps)$ are deleted and $\deg(ps)=n-1-5=n-6$. 


\item $p+j\leq n,\ p-j\leq 0$ and $p-i\leq 0\Rightarrow p< \min(i,\ n-j)$ \\
Since $j>i$ and $p+j\leq n$, we know $p+i\leq n$. Then $(p+j)s$ and $(p+i)s$ are in the signature. Thus, 4 edges adjacent to $ps$: $(is,ps),((p+i)s,ps),\ (js,ps)$ and $((p+j)s, ps)$ are deleted and $\deg(ps)=n-1-4=n-5$.  

\item $p-j< 0, \ p-i>0, \ p+i>n \Rightarrow \max(i,n-i)<p< j$ \\
Since $p+i>n$, we know $p+j>n$. Only $(p-i)s$ is in the signature. Thus, 3 edges adjacent to $ps$ are deleted and $\deg(ps)=n-1-3=n-4$.

\item $p-i< 0, \ p+i>n \Rightarrow n-i<p< i$ \\
Since $p+i>n$ and $p-i< 0$, we know $p+j>n$ and $p-j< 0$. Then, only 2 edges adjacent to $ps$: $(is,ps),\ (js,ps)$ are deleted and $\deg(ps)=n-1-2=n-3$.

\item $p-i< 0, \ p+j>n, \ p+i\leq n \Rightarrow n-j<p\leq \min(i,n-i)$ \\
Since $p-i\leq 0$, we know $p-j\leq 0$. Only $(p+i)s$ is in the signature. Thus, 3 edges adjacent to $ps$ are deleted and $\deg(ps)=n-1-3=n-4$. 

\item $p-j\leq 0,\ p-i>0, \ p+j>n, \ p+i\leq n \Rightarrow \max(i,n-j)<p\leq \min(j,n-i)$ \\
Two values $(p-i)s$ and $(p+i)s$ are in the signature. Thus, 4 edges adjacent to $ps$ are deleted and $\deg(ps)=n-1-4=n-5$. An example is illustrated in Figure~\ref{fig:delete2}.
\end{enumerate}
Combining all the cases, we have the general result for $p \notin \{2i, 2j, j-i, j+i\}$. If $p=2i, 2j$, we need to add $\epsilon_i$, $\epsilon_j$ which are defined in the previous proof. If $p=i+j$, then $(p-i)s$ and $js$ are the same vertex, and $(p-j)s$ and $is$ are the same vertex. So we overcounted the number of deleted edges by 2 and this defines $\gamma_{ij}$. If $p=j-i$, then $js$ and $(p-i)s$ are the same vertex, which defines $\beta_{ij}$. 
\end{proof}

\subsection{Structure Lemmas}
Recall that $s_1$ is the smallest element in the proper monograph signature $S = S(G)$, and $k$ is the codegree of $G$. Then \cite{bloom1984acg} make the following observation:

\begin{obs}\label{obs:smallest_arithmetic}[Fact 3.1]
The elements of $S$ can be partitioned into at most $k+1$ arithmetic progressions with common difference $s_1$. Furthermore, at least one such arithmetic progression has at least $\frac{n}{k+1}$ terms. 
\end{obs}

Note that the arithmetic progressions can be as short as one element, and one of them begins with $s_1$. 
\cite{bloom1984acg} continue by setting $l$ equal to the number of terms in the longest such arithmetic progression. 

\begin{obs}\label{obs:missing_from_[0]}[Fact 3.2]
Among the numbers $s_1, 2s_1, ..., (l-1) s_1$, at most $k$ are missing from $S$. 
\end{obs}

While the previous observations hold for proper monographs only, it is natural to extend this logic to all monographs. The fact that for any signature values $a > 0$ and $b < 0$, an edge exists between them if and only if $|a-b| = a + |b| \in S$ leads to the following companion observation for negative signature values. 

\begin{obs}\label{obs:neg_structure}
For each negative signature value $s_i \in S$, where $\deg(s_i)\neq 0$, the elements of $S^+$ can be partitioned into at most $k$ arithmetic progressions with common difference $|s_i|$. Furthermore, at least one such progression has at least $\frac nk$ terms. 
\end{obs}

\begin{proof}
Consider any maximal arithmetic progression with common difference $|s_i|$, where $a \geq 0$ is the largest member of this arithmetic progression. Then $|a-s_i| = a+|s_i|$ is not in $S$, and hence $a$ is not adjacent to $s_i$. Thus, the number of such arithmetic progressions is no more than the codegree $k$ since $k$ is defined as the maximum number of vertices which any vertex including $s_i$ is not adjacent to. 
\end{proof}


\begin{obs}\label{obs:neg_0}
If $s_1$ is the only negative signature value in $S$ and $0 \notin S$, then among the numbers $|s_1|, 2 |s_1|,...,(l-1)|s_1|$, at most $k-1$ are missing from $S$. 
\end{obs}

\begin{proof}

Assume that $s_1$ is in the longest arithmetic progression with $l$ terms and common difference $|s_1|$. Then, the $l$-term arithmetic progression is  $\{s_1, 0, |s_1|, \ldots,$ $(l-2) |s_1|\}$. This contradicts that $0 \notin S$, and so $s_1$ is not in the $l$-term arithmetic progression. Suppose $a$ is the smallest term of an $l$-term arithmetic progression $a, a + |s_1|, ..., a + (l-1) |s_1|$. By definition of codegree $k$, vertex $a$ fails to be adjacent to at most $k$ vertices, not including $a$, but specifically including $s_1$ which is not in the $l$-term arithmetic progression. Each vertex in this sequence which is adjacent to $a$ implies the existence of the label $a + t|s_1|-a = t |s_1|$.  
\end{proof}


\section{$M_6$ is the smallest non-autograph}
\label{sec:m6}

Observe that vertices in $M_6$ with the same degree are neighbors and that all vertices of $M_6$ have degree 3 or 4. Furthermore, the largest clique contained in $M_6$ is of size 4. We first eliminate the possibility that zero is in a signature for $M_6$. 

\begin{lemma} \label{lemma: m6_no_zero}
A set containing $0$ cannot be the signature for $M_6$.
\end{lemma}

\begin{proof}
Let $S=\{s_1, s_2, s_3, s_4, s_5, s_6\}$. Assume by contradiction that $0 \in S$ and $S$ is a signature for $M_6$. By Corollary~\ref{cor:neg_sig_values}, $M_6$ has at most $m=2$ negative signature values. 

\begin{enumerate}

	\item Case: $m=1$. That is, $s_1 <0$ and $s_2=0$. Note that by Observation~\ref{obs:deg_of_0}, $\deg(s_2) = 4$, and in particular, $s_2$ is adjacent to $s_3, s_4, s_5$, and $s_6$. Thus, $s_2$ is not adjacent to $s_1$, and so the only possible signature has the configuration depicted in Figure~\ref{fig:m6_0a}. Since $s_2$ and $s_1$ are not connected, it follows that $|s_2 - s_1| = s_2 + |s_1| = |s_1| \notin S$. But then, regardless of which vertex is labeled $s_5$, it must be adjacent to both $s_1$ and $s_6$. The first condition requires 
	$$
		|s_5 - s_1| = s_5 + |s_1| \in S \Rightarrow s_5 + |s_1| = s_6 \Rightarrow s_6 - s_5 = |s_1| \,.
	$$
	But the second condition then implies $|s_6 - s_5| = |s_1| \in S$, contradicting the fact that $|s_1| \notin S$. 


\begin{figure}
	\centering
	\subfigure[Case 2(a)]{
		\tikzset{
	node/.style={circle,inner sep=1mm,minimum size=0.2cm,draw,very thick,black,fill=blue!20,text=black},
	nondirectional/.style={thin,black},
}

\begin{tikzpicture}[yscale=1]
	\node [node] (v0) at (0, 1)	{$s_1$};
	\node [node] (v1) at (1, 0.5)	{?};
	\node [node] (v2) at (1, -0.5)	{$0$};
	\node [node] (v3) at (0, -1)	{?};
	\node [node] (v4) at (-1, -0.5)	{$s_6$};
	\node [node] (v5) at (-1, 0.5)	{?};
	\path [nondirectional] (v0) edge (v1);
	\path [nondirectional] (v0) edge (v5);
	\path [nondirectional] (v0) edge (v3);
	\path [nondirectional] (v3) edge (v4);
	\path [nondirectional] (v3) edge (v2);
	\path [nondirectional] (v1) edge (v2);
	\path [nondirectional] (v1) edge (v4);
	\path [nondirectional] (v1) edge (v5);
	\path [nondirectional] (v2) edge (v4);
	\path [nondirectional] (v2) edge (v5);
	\path [nondirectional] (v4) edge (v5);
\end{tikzpicture}
		\label{fig:m6_0a}
	} 
	\qquad		
	\subfigure[Case 2(b)]{
		\tikzset{
	node/.style={circle,inner sep=1mm,minimum size=0.2cm,draw,very thick,black,fill=blue!20,text=black},
	nondirectional/.style={thin,black},
}

\begin{tikzpicture}[yscale=1]
	\node [node] (v0) at (0, 1)	{$0$};
	\node [node] (v1) at (1, 0.5)	{?};
	\node [node] (v2) at (1, -0.5)	{$s_2$};
	\node [node] (v3) at (0, -1)	{$s_6$};
	\node [node] (v4) at (-1, -0.5)	{$s_1$};
	\node [node] (v5) at (-1, 0.5)	{?};
	\path [nondirectional] (v0) edge (v1);
	\path [nondirectional] (v0) edge (v5);
	\path [nondirectional] (v0) edge (v3);
	\path [nondirectional] (v3) edge (v4);
	\path [nondirectional] (v3) edge (v2);
	\path [nondirectional] (v1) edge (v2);
	\path [nondirectional] (v1) edge (v4);
	\path [nondirectional] (v1) edge (v5);
	\path [nondirectional] (v2) edge (v4);
	\path [nondirectional] (v2) edge (v5);
	\path [nondirectional] (v4) edge (v5);
\end{tikzpicture}
		\label{fig:m6_0b}
	}
	\caption{Possible configurations for $M_6$ discussed in Lemma~\ref{lemma: m6_no_zero}.}
	\label{fig:m6_0ab}
\end{figure}
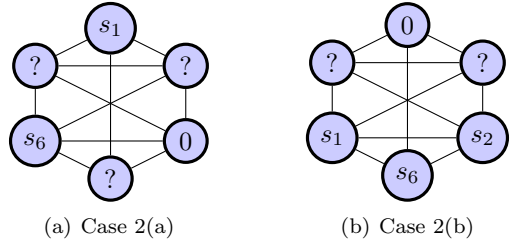



	\item Case: $m=2$. In this case $s_1 < s_2 < 0$ and $s_3 =0$. It now follows from Observation~\ref{obs:deg_of_0} that $deg(s_3) = 3$ and that in particular, $s_3$ is not adjacent to $s_1$ and $s_2$. However, it follows from Lemma~\ref{lemma:kth-largest} that $deg(s_6) =3$, and that $s_6$ is similarly not adjacent to $s_1$ and $s_2$. But this violates the structure of $M_6$, since then $s_1$ and $s_2$ are \emph{both} non-adjacent to $s_3$ and $s_6$, and no such vertex exists in $M_6$. This impossible situation is depicted in Figure \ref{fig:m6_0b}.

\end{enumerate}  
\end{proof}

\begin{lemma}\label{lemma:smallest_non_signature}
A set $P$ that contains an arithmetic progression of 5 non-negative numbers and a negative number is not a signature for $M_6$.
\end{lemma}

\begin{proof}
Assume by contradiction that $P=\{p,t,t+s,t+2s,t+3s,t+4s\}$ is a signature for $M_6$, where $t>0$, $s>0$ and $p<0$. 
We consider two cases based on whether $s$ is in $P$ or not.
\begin{enumerate}
\item Case: $s\in P$. Since $t+s \geq s > 0$, $s$ can only equal $t$ or $t+s$. If $s = t$, $ P= \{p,s,2s,3s,4s,5s\}$. If $s = t + s$, $P = \{p,0,s,2s,3s,4s\}$. The vertices with signature values $s,2s,3s,4s,5s$ (if $s=t$) or $0,s,2s,3s,4s$ (if $s = t + s$) form a clique of size 5 in $M_6$, which contradicts that the clique number of $M_6$ is 4. 

\item Case: $s\notin P$. Consider the difference set of $t+2s$. 
$$D(t+2s)=\{t+2s-p,2s,s,s,2s\}$$
Since $s\notin P$, we must have that $2s\in P$ since otherwise $\deg(t+2s) \leq 1$, and $M_6$ does not contain such a vertex. Thus, $\deg(t+2s) = 3$. Similarly, consider the difference sets of $t+3s$ and $t+s$. 
$$D(t+3s)=\{t+3s-p,3s,2s,s,s\}$$
$$D(t+s)=\{t+s-p,s,s,2s,3s\}$$
Since $s\notin P$, $\deg(t+3s)$ and $\deg(t+s)$ cannot be 4. Then, three vertices would have degree of 3: $t+3s,t+2s,t+s$. This contradicts that only two vertices of $M_6$ have degree of 3. 
\end{enumerate}
Hence, $P$ is not a signature for $M_6$.


\end{proof}

\begin{theorem}
$M_6$ is a non-autograph.
\end{theorem}

\begin{proof}
Assume by contradiction that $M_6$ has a signature $S$. As noted in Section~\ref{sec:prelim}, $M_6$ is a monograph. 

Since all vertices of $M_6$ have degree 3 or 4, $M_6$ has at most 2 negative signature values by Corollary~\ref{cor:neg_sig_values}. We consider three cases based on the number of negative signature values.
\begin{enumerate}
\item Case: $m=2$: Let $s_1<s_2<0<s_3<s_4<s_5<s_6$. Since $s_6$ is only adjacent to vertices with positive signature values, we know $\deg(s_6)=3$ and $\deg(s_1)=\deg(s_2)=4$, since $s_1,s_2$ are not adjacent to $s_6$ (see Figure~\ref{fig:m6_1}. Then, $s_1$ and $s_2$ are both adjacent to $s_5$. Since $s_5-s_1$ and $s_5-s_2$ are both greater than $s_5$ and are both in the signature, we have that
$$s_5-s_1=s_6\qquad s_5-s_2=s_6$$ 
This contradicts that $s_1<s_2$.

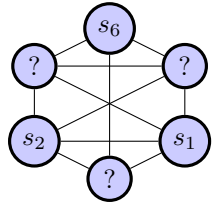
\begin{figure}[ht]
	\centering
	\tikzset{
	node/.style={circle,inner sep=1mm,minimum size=0.2cm,draw,very thick,black,fill=blue!20,text=black},
	nondirectional/.style={thin,black},
}

\begin{tikzpicture}[yscale=1]
	\node [node] (v0) at (0, 1)	{$s_6$};
	\node [node] (v1) at (1, 0.5)	{?};
	\node [node] (v2) at (1, -0.5)	{$s_1$};
	\node [node] (v3) at (0, -1)	{?};
	\node [node] (v4) at (-1, -0.5)	{$s_2$};
	\node [node] (v5) at (-1, 0.5)	{?};
	\path [nondirectional] (v0) edge (v1);
	\path [nondirectional] (v0) edge (v5);
	\path [nondirectional] (v0) edge (v3);
	\path [nondirectional] (v3) edge (v4);
	\path [nondirectional] (v3) edge (v2);
	\path [nondirectional] (v1) edge (v2);
	\path [nondirectional] (v1) edge (v4);
	\path [nondirectional] (v1) edge (v5);
	\path [nondirectional] (v2) edge (v4);
	\path [nondirectional] (v2) edge (v5);
	\path [nondirectional] (v4) edge (v5);
\end{tikzpicture}
	\caption{Possible configuration for $M_6$ with two negative signatures values ($s_1$ and $s_2$).}
	\label{fig:m6_1}
\end{figure}

\item Case: $m=1$: Let $s_1<0<s_2<s_3<s_4<s_5<s_6$. By Lemma~\ref{lemma:kth-largest}, $s_6$ is not adjacent to $s_1$. We consider two cases based on $\deg(s_1)$. 
	\begin{enumerate}
	\item Case: $\deg(s_1)=4$. Since $s_1$ is not adjacent to $s_6$, we know $\deg(s_6)=3$. Note that $s_1$ is adjacent to $s_2, s_3, s_4$ and $s_5$. Thus, $s_5-s_1>s_5$ has to be $s_6$ and $s_6-s_5=|s_1|$ (see Figure~\ref{fig:m6_2a}). Since $s_4<s_4-s_1<s_5-s_1=s_6$, we have that $s_4-s_1=s_5 \Rightarrow s_5-s_4=|s_1|$. Since $s_3<s_3-s_1<s_4-s_1=s_5$, we have that $s_3-s_1=s_4\Rightarrow s_4-s_3=|s_1|$. Since $s_2<s_2-s_1<s_3-s_1=s_4$, we have that $s_2-s_1=s_3\Rightarrow s_3-s_2=|s_1|$. Therefore, $s_2,s_3,s_4,s_5,s_6$ form an arithmetic progression of length 5 with common difference $|s_1|$.
	$$s_3-s_2=s_4-s_3=s_5-s_4=s_5-s_4=-s_1$$
    By Lemma~\ref{lemma:smallest_non_signature}, $S$ is not a signature for $M_6$.
    
\begin{figure}
	\centering
	\subfigure[Case 2(a)]{
		\tikzset{
	node/.style={circle,inner sep=1mm,minimum size=0.2cm,draw,very thick,black,fill=blue!20,text=black},
	nondirectional/.style={thin,black},
}

\begin{tikzpicture}[yscale=1]
	\node [node] (v0) at (0, 1)	{?};
	\node [node] (v1) at (1, 0.5)	{$s_1$};
	\node [node] (v2) at (1, -0.5)	{?};
	\node [node] (v3) at (0, -1)	{$s_6$};
	\node [node] (v4) at (-1, -0.5)	{?};
	\node [node] (v5) at (-1, 0.5)	{?};
	\path [nondirectional] (v0) edge (v1);
	\path [nondirectional] (v0) edge (v5);
	\path [nondirectional] (v0) edge (v3);
	\path [nondirectional] (v3) edge (v4);
	\path [nondirectional] (v3) edge (v2);
	\path [nondirectional] (v1) edge (v2);
	\path [nondirectional] (v1) edge (v4);
	\path [nondirectional] (v1) edge (v5);
	\path [nondirectional] (v2) edge (v4);
	\path [nondirectional] (v2) edge (v5);
	\path [nondirectional] (v4) edge (v5);
\end{tikzpicture}
		\label{fig:m6_2a}
	} 
	\qquad		
	\subfigure[Case 2(b)]{
		\tikzset{
	node/.style={circle,inner sep=1mm,minimum size=0.2cm,draw,very thick,black,fill=blue!20,text=black},
	nondirectional/.style={thin,black},
}

\begin{tikzpicture}[yscale=1]
	\node [node] (v0) at (0, 1)	{$s_1$};
	\node [node] (v1) at (1, 0.5)	{?};
	\node [node] (v2) at (1, -0.5)	{?};
	\node [node] (v3) at (0, -1)	{?};
	\node [node] (v4) at (-1, -0.5)	{$s_6$};
	\node [node] (v5) at (-1, 0.5)	{?};
	\path [nondirectional] (v0) edge (v1);
	\path [nondirectional] (v0) edge (v5);
	\path [nondirectional] (v0) edge (v3);
	\path [nondirectional] (v3) edge (v4);
	\path [nondirectional] (v3) edge (v2);
	\path [nondirectional] (v1) edge (v2);
	\path [nondirectional] (v1) edge (v4);
	\path [nondirectional] (v1) edge (v5);
	\path [nondirectional] (v2) edge (v4);
	\path [nondirectional] (v2) edge (v5);
	\path [nondirectional] (v4) edge (v5);
\end{tikzpicture}
		\label{fig:m6_2b}
	}
	\caption{Possible configurations for $M_6$ in the case where $s_1$ is the only negative signature value. At left, $\deg(s_1) =4$, while at right, $\deg(s_1) = 3$. }
	\label{fig:m6_2ab}
\end{figure}
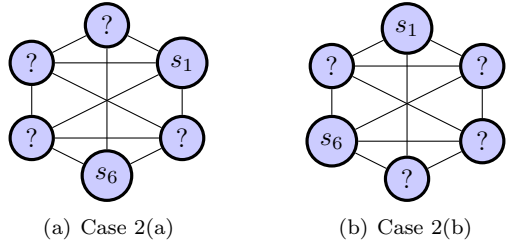
    
    \item Case: $\deg(s_1)=3$. Since $s_1$ is not adjacent to $s_6$, we know $\deg(s_6)=4$ and $s_6$ is adjacent to all vertices with positive signature values. By Lemma~\ref{lemma:max}, 
    $$s_6=s_2+s_5=s_3+s_4\Rightarrow s_6-s_5=s_2, \ s_5-s_4=s_3-s_2$$
    We consider 4 cases based on the other vertex to which $s_1$ is not adjacent (see Figure~\ref{fig:m6_2b}).
     
    	\begin{enumerate}
    	\item $s_5$ is not adjacent to $s_1$. Then, $\deg(s_5)=4$. Consider the intersection of the difference set of $s_5$ and $S$, $D(s_5)\cap S$.
    	$$D(s_5)\cap S=\{s_6-s_5, s_5-s_2, s_5-s_3, s_5-s_4\}$$
    	We already know that $s_6-s_5=s_2$. Note that all other elements of $D(s_5)\cap S$ are less than $s_5$ and greater than $s_1$, and thus can only take values of $s_2,s_3$ or $s_4$. Therefore, we can find a one-to-one correspondence between sets $D(s_5)\cap S$ and $\{s_2,s_3,s_4\}$.
    	$$s_5-s_2=s_4 \qquad s_5-s_3=s_3 \qquad s_5-s_4=s_2$$ 
    	Since $s_5-s_4=s_2$, we have that $s_3-s_2=s_2 \Rightarrow s_3=2s_2$. Then $s_5=2s_3=4s_2$, $s_6=s_5+s_2=5s_2$ and $s_4=s_5-s_2=3s_2$. Therefore, $s_2,s_3,s_4,s_5,s_6$ form an arithmetic progression of length 5 with common difference $s_2$.
    	By Lemma~\ref{lemma:smallest_non_signature}, $S$ is not a signature for $M_6$.
    	
    	\item $s_4$ is not adjacent to $s_1$. Then, $\deg(s_4)=4$. Consider the intersection of the difference set of $s_4$ and $S$, $D(s_4)\cap S$.
    	$$D(s_4)\cap S=\{s_4-s_3, s_4-s_2, s_6-s_4, s_5-s_4\}$$
    	Note that $s_4-s_3$ and $s_4-s_2$ are both less than $s_4$ and greater than $s_1$. Thus, we must have $s_4-s_3=s_2$ and $s_4-s_2=s_3$. We already know that $s_6-s_4=s_3$. Then, $s_5-s_4<s_6-s_4$ can only be $s_2$. Since $s_5-s_4=s_3-s_2$, $s_3=2s_2$. Thus, $s_4=3s_2$, $s_5=4s_2$ and $s_6=s_4+s_3=5s_2$.
    	Note that $s_2,s_3,s_4,s_5,s_6$ form an arithmetic progression with common common difference $s_2$. By Lemma~\ref{lemma:smallest_non_signature}, $S$ is not a signature for $M_6$.
    	
    	\item $s_3$ is not adjacent to $s_1$. Then, $\deg(s_3)=4$. Consider the intersection of the difference set of $s_3$ and $S$, $D(s_3)\cap S$.
    	 $$D(s_3)\cap S=\{s_3-s_2, s_4-s_3, s_5-s_3, s_6-s_3\}$$
    	Since $s_3-s_2<s_3$, $s_3-s_2$ has to be $s_2$. We already know that $s_6-s_3=s_4$. Note that both $s_5-s_3$ and $s_4-s_3$ are less than $s_6-s_3=s_4$ and greater than $s_3-s_2=s_2$. Then $s_5-s_3$ and $s_4-s_3$ both are $s_3$, which contradicts that $s_4\neq s_5$.
    	
    	\item $s_2$ is not adjacent to $s_1$. Then, $\deg(s_2)=4$. Consider the intersection of the difference set of $s_2$ and $S$, $D(s_2)\cap S$.
    	$$D(s_2)\cap S=\{s_3-s_2, s_4-s_2, s_5-s_2, s_6-s_2\}$$
    	We already know that $s_6-s_2=s_5$. All other elements of $D(s_2)\cap S$ are less than $s_5$ and greater than $s_1$, and thus can only take values of $s_2,s_3$ or $s_4$. Therefore, we can find a one-to-one correspondence between sets $D(s_2)\cap S$ and $\{s_2,s_3,s_4\}$.
    	$$s_3-s_2=s_2 \qquad s_4-s_2=s_3 \qquad s_5-s_2=s_4$$
    	Thus, $s_3=2s_2$, $s_4=3s_2$, $s_5=4s_2$ and $s_6=s_5+s_2=5s_2$. Therefore, $s_2,s_3,s_4,s_5,s_6$ form an arithmetic progression of length 5 with common difference $s_2$. By Lemma~\ref{lemma:smallest_non_signature}, $S$ is not a signature for $M_6$.
    	\end{enumerate}
    	
	\end{enumerate}
	\item Case: $m=0$. Let $s$ denote the smallest signature value in $S$. By Observation~\ref{obs:smallest_arithmetic}, $S$ can be partitioned into at most 3 maximal arithmetic progressions of common difference $s$. $S$ cannot consist of 1 arithmetic progression with common difference $s$ since $[1]^n_s$ is a signature for the complete graph. Thus, we consider two cases. 
	\begin{enumerate}
	\item Case: 3 arithmetic progressions. Let the three progressions be $[1]_s^{l_0}$, $[a_1]_s^{l_1}$, and $[a_2]_s^{l_2}$, where we assume without loss of generality that $s < a_1 < a_2$. Note that $s$ is not adjacent to either $a_1$ or $a_2$ since $[a_1]_s^{l_1}$ and $[a_2]_s^{l_2}$ are maximal progressions. Thus, $\deg(s)=3$ and $\deg(a_1)=\deg(a_2)=4$. Note that $s$ is not in the complete subgraph on 4 vertices. Then, the length of $[1]_s^{l_0}$, $l_0$ must be less than 4. Moreover, $l_1$ and $l_2$ are less than 4. Assume by contradiction that $l_i=4$, where $i=1,2$. Then, the other two arithmetic progressions have length 1, and there exists an arithmetic progression of 4 terms: $\{a_i,a_i+s,a_i+2s,a_i+3s\}$. Since $a_i$ is not adjacent to $s$ only, $a_i+s-a_i,a_i+2s-a_i,a_i+3s-a_i\in S\Rightarrow s,2s,3s\in S$, which contradicts that $l_0=1$.
	Therefore, $l_0,l_1,l_2$ are all less than 4. We discuss 3 cases of $l_0$.
	\begin{enumerate}
	\item $l_0=1$. Since $l_1,l_2$ are both less than 4, $S=\{s,a_1,a_1+s,a_2,a_2+s,a_i+2s\}$, where $i$ is 1 or 2. However, $|a_i+2s-a_i|=2s\notin S$, contradicting that $a_i$ is adjacent to every vertex besides $s$. 
	
	\item $l_0=3$. Then $S=\{s,2s,3s, a_1,a_2,a_i+s\}$, where $i=1$ or 2. Since $2s,3s$ are adjacent, $2s$ and $3s$ both have degree of 4, and $\deg(a_i+s)=3$. Note that $a_i+s$ is not adjacent to $2s,3s$. Since $a_1-2s<a_1$ and $a_1$ is adjacent to $2s$, then $a_1-2s \in \{s,2s,3s\}$. Since $\{s,2s,3s\}\in S$, $a_1=5s$. If $a_i+s=a_1+s=6s$, then $a_i+s$ is adjacent to $3s$, a contradiction. Thus, $l_1=1$, $l_2=2$ and $S=\{s,2s,3s,5s,a_2,a_2+s\}$. Since $a_2-2s<a_2\in S$, $a_2-2s=a_1=5s\Rightarrow a_2=7s$. However, $a_2-3s=4s\notin S$, which contradicts that $3s$ is adjacent to every vertex besides $a_2+s$.
	
	\item $l_0=2$. Note that $a_1$ is adjacent to $2s$.  We discuss 3 cases based on $l_1,l_2$. 
		\begin{enumerate}
		\item $l_1=2, \ l_2=2$. Then $S=\{s,2s,a_1,a_2,a_1+s,a_2+s\}$. Since $a_2$ is adjacent to every vertex besides $s$, $a_2-a_1$, $a_2+s-a_1$, $a_2-(s+a_1)\in S$. They form an arithmetic progression of length 3 with common difference $s$, contradicting that $l_0=l_1=l_2=2$.
		\item $l_1=1,\ l_2=3$.  Then $S=\{s,2s,a_1,a_2,a_2+s,a_2+2s\}$, and $a_2-a_1, a_2+s-a_1, a_2+2s-a_1\in S$. They form an arithmetic progression of length 3 with common difference $s$, and $S$ only contains one arithmetic progression of 3 terms: $\{a_2,a_2+s,a_2+2s\}$. Thus, $a_2-a_1=a_2$, a contradiction.
		\item $l_1=3,\ l_2=1$. Then $S=\{s,2s,a_1,a_2,a_1+s,a_1+2s\}$ and $a_1+s-a_2, a_1+2s-a_2\in S$. We know $\{s,2s\}$ is the only arithmetic progression of two terms in $S$. Then $a_1+s-a_2=s$, a contradiction. 
		\end{enumerate}
	\end{enumerate}
\item Case: 2 arithmetic progressions. Let the two maximal progressions be $[1]_s^{l_0}, [a_1]_s^{l_1}$. $$S=\{s,2s,...,l_0s\}\cup \{a_1, a_1+s,\ldots,a_1+(l_1-1)s\}$$
Consider the difference set of $s$, $D(s)$.
$$D(s)=\{a_1-s,a_1,a_1+s,...,a_1+(l_1-2)s\}\cup \{s,2s,...,(l_0-1)s\}$$
Since $[a_1]_s^{l_1}$ is a maximal arithmetic progression, $a_1-s$ cannot be in the signature. Note that all terms in $D(s)$ besides $a_1-s$ are in the signature, and $\deg(s)= 4$. Since $s$ is only non-adjacent to $a_1$, $\deg(a_1)=3$. We discuss 5 cases based on $l_0$ and $l_1$. 
	\begin{enumerate}
	\item $l_0=1 \ l_1=5$. Then, $S=\{s,a_1,a_1+s,a_1+2s,a_1+3s,a_1+4s\}$. Consider the difference set of $a_1+2s$. 
	$$D(a_1+2s)=\{s,s,2s,2s,a_1+s\}$$
	Thus, $\deg(a_1+2s)=3$. However, $a_1+2s$ is not adjacent to $a_1$, which contradicts that $a_1$ and $a_1+2s$ both have degree 3.
	
	\item $l_0=2,\ l_1=4$. Then, $S=\{s,2s,a_1,a_1+s,a_1+2s,a_1+3s\}$. Consider the difference set of $a_1+2s$.
	$$D(a_1+2s)=\{s,s,2s,a_1,a_1+s\}$$
	All elements of $D(a_1+2s)\in S$, which contradicts that all vertices have degree 3 or 4.
	
	\item $l_0=3,\ l_1=3$. Then, $S=\{s,2s,3s,a_1,a_1+s,a_1+2s\}$. Consider the difference set of $2s$.
	$$D(2s)=\{s,s,a_1-2s,a_1-s,a_1\}$$
	We know that $a_1-s\notin S$. If $a_1-2s\notin S$, then $\deg(2s)=3$. However, $2s$ is not adjacent to $a_1$, the other vertex with degree 3, a contradiction. Thus, $a_1-2s\in S$. Since $a_1-2s<a_1$ and $a_1>4s$, $a_1-2s$ can only be $3s$. Consider the difference set of $a_1=5s$.
		$$D(a_1)=\{4s,3s,2s,s,2s\}$$
	Only $4s$ in $D(a_1)\notin S$, contradicting that $\deg(a_1)=3$.
	
	\item $l_0=4, \ l_1=2$. Then, $S=\{s,2s,3s, 4s,a_1,a_1+s\}$. Consider the difference set of $2s$. Following the same argument, we deduce that $a_1-2s$ can only be $4s\Rightarrow a_1=6s$. Consider the difference set of $a_1=6s$.
	$$D(a_1)=\{5s,4s,3s,2s,s\}$$ 
	Only $5s$ in $D(a_1)$ is not in $S$, contradicting that $\deg(a_1)=3$.
	
	\item $l_0=5,\ l_1=1$. Then, $S=\{s,2s,3s, 4s,5s,a_1\}$. Consider the different sets of $2s,3s,4s,5s$. 
	$$D(2s)=\{s,s,2s,3s,a_1-s\}$$
	$$D(3s)=\{2s,s,s,2s,a_1-3s\}$$
	$$D(4s)=\{3s,2s,s,s,a_1-4s\}$$
	$$D(5s)=\{4s,3s,2s,s,a_1-5s\}$$
	Thus, $s, 2s,3s,4s,5s$ all have degree 4, which contradicts that 2 vertices have degree 4.
	\end{enumerate} 
	\end{enumerate}
\end{enumerate}
We have exhausted all possible signatures for $M_6$, without finding one. This completes the proof. 
\end{proof}

\section{$G_n$ are non-autographs}
\label{sec:g_n}

Let $G = \overline{2 C_{\frac{n}{2}}}$, for some even $n\geq 8$. That is, let $G_n$ consist of the graph formed by taking the complement of two disjoint cycles of length $n/2$. A symmetric depiction of $G_n$ for $n=10$ is shown in Figure \ref{fig:G}, with which we also include a drawing of its complement for clarity. 

We begin by making three observations about the structure of $G_n$ that will prove useful later on:

\begin{obs}\label{obs:clique}
The clique number of $G_n$ is at most $\frac{n}{2}$, where $n$ is even. 
\end{obs}

\begin{proof}
Consider cycles of length $\frac n2$ in $\overline{G_n}$. If $4 \mid n$, the maximum clique includes every other vertex in each cycle. If $4 \nmid n$, each cycle in $\overline{G_n}$ has an odd length, and so the maximum clique has $\frac{n}{2}-1$ vertices. In both cases, the clique number is no more than $\frac{n}{2}$.
\end{proof}

\begin{obs}\label{obs:independent_set}
The size of the largest independent set in $G_n$ is 2, where $n$ is even. 
\end{obs}

\begin{proof}
Note that $\overline{G_n}$ contains two cycles of length greater than or equal to 5. Every vertex is non-adjacent to exactly two vertices. Those two vertices are always adjacent to avoid forming a 3-cycle in $\overline{G_n}$. 
\end{proof}

\begin{obs}\label{obs:longest_path}
The length of the longest path in $\overline{G_n}$ is $\frac{n}{2}-1$, where $n$ is even. 
\end{obs}

\begin{proof}
Take one of the cycles in $\overline{G_n}$ and remove one edge. 
\end{proof}

We first prove a helpful lemma based on our previous results on arithmetic progressions. 


\begin{lemma} \label{lemma: gn_no_zero}
A set containing $0$ cannot be the signature for $G_n$ for some even $n\geq 8$.
\end{lemma}

\begin{proof}
Note that $\delta(G_n) = n-3$, and the codegree of $G_n$ is $k = 2$. By Corollary~\ref{cor:neg_sig_values} that there are at most 2 negative signature values.
Assume by contradiction that $0\in S=\{s_1,s_2,\ldots,s_n\}$ and $S$ is the signature for $G_n$. By Observation~\ref{obs:deg_of_0}, it is clear that if $S$ contains $0$, it must contain exactly two negative signature values, since otherwise $\deg(0) > n-2$. Thus, since the vertex labeled with $0$ is adjacent to all $n-3$ vertices with positive signature values, it is not adjacent to $s_1$ and $s_2$. Specifically,  Then $s_1<s_2<0$, $s_3=0$, and $0<s_4<s_5< \ldots < s_n$. By Lemma~\ref{lemma:kth-largest}, $s_n$ is also not adjacent to $s_1$ and $s_2$. Thus, $s_1, s_n, s_2$ and $s_3$ form a 4-cycle in $\overline{G_n}$. This is only possible when $n=8$, shown in Figure~\ref{fig:g8}. In this case, $s_7$ is adjacent to both $s_1$ and $s_2$. 
$$s_7-s_1 > s_7 \in S \qquad s_7 - s_2 > s_7 \in S \Rightarrow s_7-s_1 = s_7-s_2 = s_8$$
This contradicts the fact that $G_n$ is a monograph. 
\end{proof}

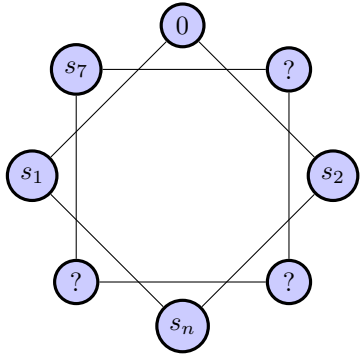
\begin{figure}[ht]
	\centering
	\tikzset{
	node/.style={circle,inner sep=1mm,minimum size=0.2cm,draw,very thick,black,fill=blue!20,text=black},
	nondirectional/.style={thin,black},
}

\begin{tikzpicture}[scale=2]
	\def\x{0.7071068}
	\node [node] (v0) at (0, 1)	{$0$};
	\node [node] (v1) at (\x, \x)	{?};
	\node [node] (v2) at (1, 0)	{$s_2$};
	\node [node] (v3) at (\x, -\x)	{?};
	\node [node] (v4) at (0, -1)	{$s_n$};
	\node [node] (v5) at (-\x, -\x)	{?};
	\node [node] (v6) at (-1, 0)	{$s_1$};
	\node [node] (v7) at (-\x, \x)	{$s_7$};
	\path [nondirectional] (v0) edge (v2);
	\path [nondirectional] (v2) edge (v4);
	\path [nondirectional] (v4) edge (v6);
	\path [nondirectional] (v6) edge (v0);
	\path [nondirectional] (v1) edge (v3);
	\path [nondirectional] (v3) edge (v5);
	\path [nondirectional] (v5) edge (v7);
	\path [nondirectional] (v7) edge (v1);
\end{tikzpicture}
	\caption{$\overline{G_8}$: In $G_8$, $s_7$ is adjacent to both $s_1$ and $s_2$, creating the impossible situation in which there are two distinct signature values greater than $s_7$.}
	\label{fig:g8}
\end{figure}

\begin{lemma}\label{lemma:K_n-delete-1}
The set $S = \{s, 2s, ..., ns\} \setminus \{is\}$ is not a signature for $G_{n-1}$, where $n\geq 9$ and $n$ is odd. 
\end{lemma}

\begin{proof}
We can immediately apply Lemma \ref{lemma:one_deletion_degrees} to obtain the degree of vertices in the monograph generated by $S$. Note that every vertex of $G_{n-1}$ has degree $n-4$. Only for $p\in (i,n-i]$ and $p\neq 2i$, $\deg(ps)=n-4$. Clearly, $(i,n-i]\setminus \{2i\}$ containing $n-2i-1$ terms does not include all of $[1,n]\setminus \{i\}$ containing $n-1$ terms. Therefore, some vertices in $S$ do not have degree $n-4$, and $S$ cannot be a signature for $G_{n-1}$.
\end{proof}

\begin{lemma}\label{lemma:K_n-delete-2}
The signature $S = \{s, 2s, ..., ns\} \setminus \{is,js\}$ for some $i < j$ is not a signature for $G_{n-2}$, where $n\geq 10$ and $n$ is even. 
\end{lemma}

\begin{proof}
Assume by contradiction that $S$ is a signature for $G_{n-2}$. We apply Lemma~\ref{lemma:two_deletions_degree} to obtain the degree of vertices in the monograph generated by $S$. Since all vertices of $G_{n-2}$ have degree $n-5$, there does not exist $p$ such that $\max(i,n-i)<p< j$ or $n-j<p\leq \min(n-i,i)$ or $n-i<p<i$, since this would already imply too large of a degree for $ps$. Thus, $n-i\geq j \Rightarrow n\geq i+j$ if all vertices have degree $n-5$. 

If there exists $p$ such that $j<p\leq n-j$, then there must exist only one such value $p$ in $(j,n-j]$ and $\gamma_{ij}=2$ for the vertex with signature value $ps$ so that $\deg(ps)=n-5$. Therefore, it must be the case that $j$ and $n-j$ are consecutive integers, and in particular, $j+1 = n-j$. Thus $p=i+j=j+1=n-j\Rightarrow n=2j+1$, which contradicts that $n$ is even. Therefore, there does not exist $p$ such that $j<p\leq n-j$, and $\gamma_{ij}=0$ for all vertices. Thus, $i+j>n$, which contradicts the previous conclusion that $n\geq i+j$. 

\end{proof}

Our main result follows:

\begin{theorem}
Let $G_n = \overline{2 C_{\frac{n}{2}}}$, for some even $n \geq 8$. Then $G_n$ is not an autograph. 
\end{theorem}

\begin{proof}
First, note that no two vertices in $G_n$ have the same set of neighbors, and so by Observation~\ref{diff_neigobor}, $G_n$ is a monograph. Next, note that $\delta(G_n) = n-3$, and the codegree of $G_n$ is $k = 2$. It follows from Corollary~\ref{cor:neg_sig_values} that $m \leq 2$. We thus consider three cases.
  
\begin{figure}
	\centering
	\subfigure[Case 2(a)]{
		\tikzset{
	node/.style={circle,inner sep=1mm,minimum size=0.5cm,draw,very thick,black,fill=blue!20,text=black},
	nondirectional/.style={thin,black},
}

\begin{tikzpicture}[yscale=1]
	\node [node] (v0) at (0, 1)	{$s_n$};
	\node [node] (v1) at (1, 0.5)	{$s_2$};
	\node [node] (v2) at (1, -0.5)	{$s_{n-1}$};
	\node [node] (v3) at (0, -1)	{?};
	\node [node] (v4) at (-1, -0.5)	{?};
	\node [node] (v5) at (-1, 0.5)	{$s_1$};
	\path [nondirectional] (v5) edge (v1);
	\path [nondirectional] (v5) edge (v2);
	\path [nondirectional] (v0) edge (v2);
	\path [nondirectional] (v0) edge (v3);
	\path [nondirectional] (v0) edge (v4);
	\path [bend left,dotted] (v3) edge (v4);
\end{tikzpicture}
	} 
	\qquad		
	\subfigure[Case 2(b)]{
		\tikzset{
	node/.style={circle,inner sep=1mm,minimum size=0.5cm,draw,very thick,black,fill=blue!20,text=black},
	nondirectional/.style={thin,black},
}

\begin{tikzpicture}[yscale=1]
	\node [node] (v0) at (0, 1)	{$s_n$};
	\node [node] (v1) at (1, 0.5)	{$s_2$};
	\node [node] (v2) at (1, -0.5)	{?};
	\node [node] (v3) at (0, -1)	{?};
	\node [node] (v4) at (-1, -0.5)	{$s_{n-1}$};
	\node [node] (v5) at (-1, 0.5)	{$s_1$};
	\path [nondirectional] (v5) edge (v1);
	\path [nondirectional] (v1) edge (v4);
	\path [nondirectional] (v0) edge (v4);
	\path [nondirectional] (v0) edge (v2);
	\path [nondirectional] (v0) edge (v3);
	\path [bend right,dotted] (v3) edge (v2);
\end{tikzpicture}
	}
	\caption{ Possible configurations for $G_n$ with $m=2$. At left, $s_{n-1}$ is adjacent to $s_1$ but not $s_2$. At right, the opposite. }
	\label{fig:gn_51}
\end{figure}
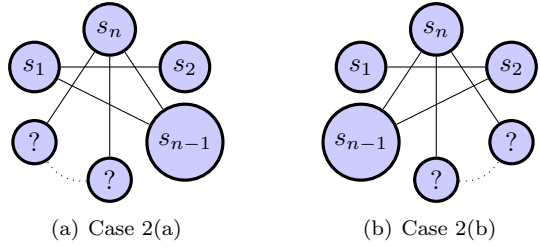

\begin{enumerate}
\item Case: $m=2$: Let $s_1 < s_2 < 0$. Note that the maximum signature value $s_n$ is adjacent to all vertices with positive signature values, namely $s_3, s_4,..., s_{n-1}$. By Lemma~\ref{lemma:max}, 
$$
s_n=s_3+s_{n-1} \quad\Rightarrow \quad s_n-s_{n-1}=s_3
$$
Since $s_n$ is not adjacent to either $s_1$ or $s_2$ and $s_1$, $s_2$ have different neighborhoods, it follows that $s_1$ and $s_2$ are adjacent. Likewise, $s_{n-1}$ is adjacent to either $s_1$ or $s_2$, but not both. Consider two cases where $s_{n-1}$ is adjacent to $s_2$ or $s_1$, as depicted in Figure~\ref{fig:gn_51}.
\begin{enumerate}
\item Case: $s_{n-1}$ is adjacent to $s_1$. Then $s_{n-1}-s_1>s_{n-1}$ must be $s_n$. This implies that $s_{n-1}-s_1=s_n$, and so $s_n-s_{n-1}=-s_1$. But now,
$$
	s_2-s_1=s_2+s_n-s_{n-1}<s_n-s_{n-1}=s_3 \,.
$$
Since $s_3$ is the smallest positive signature value, $s_2-s_1<s_3$ indicates that $s_2-s_1\notin S$, contradicting that $s_2$ and $s_1$ are adjacent.

\item Case: $s_{n-1}$ is adjacent to $s_2$. Then,
$$s_{n-1}-s_2=s_n$$
Since $s_n=s_3+s_{n-1}$ by Lemma~\ref{lemma:max}, it follows that $s_2=-s_3$. 

By Observation~\ref{obs:neg_structure}, $\{s_3, s_4,...,s_n\}$ can be partitioned into at most two arithmetic progressions with common difference $|s_2|=s_3$. 
If $\{s_3, s_4,...,s_n\}$ is an arithmetic progression with common difference $s_3$, the graph would contain a complete subgraph of $n-2$ vertices, contradicting that the clique number of $G_n$ cannot exceed $n/2$ (Observation~\ref{obs:clique}). Thus, $\{s_3, s_4,...,s_n\}$ consists of two arithmetic progressions with common difference $s_3$. Let the two arithmetic progressions be 
$$
\{s_3, 2s_3,...,l_0s_3\}\text{ and }\{a_1, a_1+s_3,...,s_n\}
$$
where $a_1\neq (l_0+1)s_3$. Note that $s_1$ is adjacent to all vertices other than $s_n$ and $s_{n-1}$. In particular, $s_1$ is adjacent to $l_0s_3$. Then, $l_0s_3-s_1=(l_0+2)s_3$ is in the signature and $(l_0+2)s_3\in \{a_1,...s_n\}$. Thus, $\{a_1,...,s_n\}$ consists of multiples of $s_3$ and $a_1>(l_0+1)s_3$. \\
Since $|s_1|>|s_2|=s_3$, $|s_1|>s_n-s_{n-1}\Rightarrow s_{n-1}+|s_1|>s_n$. Thus, $|s_{n-2}-s_1|=s_{n-2}+|s_1|=s_{n-1}-s_3+|s_1|>s_n-s_3=s_{n-1}$. Since $s_{n-2}$ is adjacent to $s_1$, $s_{n-2}+|s_1|$ must be $s_n$. Therefore, $$s_n-s_{n-2}=|s_1|=2s_3$$
Since $a_1>(l_0+1)s_3$, $a_1+s_3>(l_0+2)s_3$ and all terms in $\{a_1,a_1+s_3,...,s_n\}$ except for $a_1$ are greater than $(l_0+2)s_3$. It follows that $a_1=(l_0+2)s_3$. Therefore, the signature $S$ is
$$S=\{-2s_3,-s_3\}\cup\{s_3,2s_3,...,l_0s_3\}\cup\{(l_0+2)s_3, (l_0+3)s_3,...,s_n\}$$

Consider the difference set of $s_3$, $D(s_3)$.
\nonumber
\begin{multline}
D(s_3)=\{3s_3,2s_3\}\cup\{s_3,2s_3,...,(l_0-1)s_3\}\\
\cup \{(l_0+1)s_3,(l_0+2)s_3,...s_n-s_3\}
\end{multline}
\nonumber
Observe that in $D(s_3)$ only $(l_0+1)s_3$ is not in the signature. Thus, $\deg(s_3)=n-2$, which contradicts that every vertex in $G_n$ has degree $n-3$.
\end{enumerate}

\item Case: $m=1$. Let $s_1<0$. Observation~\ref{obs:neg_structure} implies that $S$ consists of $s_1$ and at most two arithmetic progressions with common difference $|s_1|$. Two cases are discussed based on the number of arithmetic progressions in $S$. Recall that $l$ denotes the length of the longest arithmetic progression. 

\begin{enumerate}
\item Case: 1 arithmetic progression \\
$$S=\{s_1\}\cup \{s_2, s_2+|s_1|,\ldots, s_2+(l-1)|s_1|\}$$
$$D(s_1)=\{s_2+|s_1|, s_2+2|s_1|,\ldots s_2+(l-1)|s_1|,s_2+l|s_1|\}$$
In $D(s_1)$, only $s_2+l|s_1|\notin S$ and $s_1$ is only non-adjacent to $s_2+(l-1)|s_1|$, contradicting that every vertex of $G_n$ has degree $n-3$.

\item Case: 2 arithmetic progressions \\
By Observation \ref{obs:neg_0}, at most one of $\{|s_1|, 2|s_1|, ..., (l-1)|s_1|\}$ is not in $S$. If $|s_1|\notin S$, we suppose the two arithmetic progressions of $S$ are $\{2|s_1|,3|s_1|,\ldots,(l_1+1)|s_1|\}$ and $\{s_n-(l_2-1)|s_1|,...,s_n-|s_1|,s_n\}$. Since we assume $|s_1|\notin S$, $\{2|s_1|,3|s_1|,\ldots,(l_1+1)|s_1|\}$ and $\{s_1,s_n-(l_2-1)|s_1|,...,s_n-|s_1|,s_n\}$ form two paths in $\overline{G_n}$, starting with $2|s_1|,s_1$ and ending with $(l_1+1)|s_1|,s_n-(l_2-1)|s_1|$ respectively. This contradicts that $\overline{G_n}$ consists of two cycles. Therefore, $|s_1|\in S$ and $|s_1|$ is the smallest term in one arithmetic progression.
 
Suppose the smallest term of the other arithmetic progression is $s_j$. If no element from $\{|s_1|, 2|s_1|,\ldots, (l_1-1)|s_1|\}$ is missing, we suppose the length of the other arithmetic progressions is $l_2$. 
\end{enumerate} 
\begin{align*}
S &=\{s_1\}  \cup\{|s_1|, 2|s_1|,\ldots, (l-1)|s_1|\}\\
	&\qquad \cup\{s_j, s_j+|s_1|,\ldots, s_j+(l_2-1)|s_1|\} \\
D(|s_1|) &=\{2|s_1|\}   \cup\{|s_1|,2|s_1|,\ldots,(l-2)|s_1|\}\\
	&\qquad \cup\{|s_j-|s_1||,s_j, s_j+|s_1|,\ldots,(l_2-2)|s_1|\}
\end{align*}
In $D(|s_1|)$, only $|s_j-|s_1||$ may not be in $S$ and $\deg(|s_1|)\geq n-2$. This contradicts every vertex of $G_n$ has degree $n-3$. 

Thus, exactly one member of $\{|s_1|, 2|s_1|, ..., (l-1)|s_1|\}$ is not in $S$. Let $t|s_1|$ be that value. Since we know $|s_1|\in S$, it remains to check two possibilities:
			\begin{enumerate}
				\item Case: $2 \leq t \leq l-2$: In this case
				$$
					S = \{s_1\} \cup \{|s_1|, 2|s_1|,...,(t-1)|s_1|\} \cup \{(t+1)|s_1|,(t+2)|s_1|,..., n|s_1|\}
				$$ 
				This is also not a signature for $G_n$, since the vertex labeled $(n-1)|s_1|$ is adjacent to every other vertex except for $(n-1-t)|s_1|$. This contradicts that every vertex in $G_n$ has degree $n-3$. 
				\item Case: $t=l-1$.
				In this case there exists a maximal arithmetic progression of length exactly $l-2$, namely $\{|s_1|,...,(l-2)|s_1|\}$. By definition of $l$, there also exists a maximal arithmetic progression of length $l$ with common difference $|s_1|$. By Observation~\ref{obs:deg_of_0}, $0\notin S$, and so $s_1$ is not in any arithmetic progression with common difference $|s_1|$. 
				The first term of the arithmetic progression of length $l$ is greater than $s_1$, and its last term is greater than $(l-2)|s_1|$. 
				Note that the two arithmetic progressions are disjoint. If they have a common term, then all terms following the common term are the same since both arithmetic progressions have common difference $|s_1|$. Then $\{|s_1|,...,(l-2)|s_1|\}$ can be extended to the last term of the $l$-term arithmetic progression, and would not be a maximal arithmetic progression. Thus, we must have that $l + l-2 = n-1 \Rightarrow l = \frac{n+1}{2}$, contradicting that $n$ is even. 
			\end{enumerate}

\item Case: $m=0$: 
Let $s$ denote the smallest signature value in $S$. By Observation~\ref{obs:smallest_arithmetic}, $S$ can be partitioned into at most 3 maximal arithmetic progressions of common difference $s$, one of which begins with $s$. $S$ cannot consist of 1 arithmetic progression with common difference $s$ since $[1]_{s}^{n}$ is a signature for the complete graph $K_n$. 
\begin{enumerate}
\item Case: 2 arithmetic progressions. Let the two maximal progressions be $[1]_s^{l_0}, [a_1]_s^{l_1}$.  
$$S=\{s,2s,...,l_0s\}\cup \{a_1, a_1+s,...,a_1+(l_1-1)s\}$$
Consider the difference set of $s$, $D(s)$.
$$D(s)=\{a_1-s,a_1,a_1+s,...,a_1+(l_1-2)s\}\cup \{s,2s,...,(l_0-1)s\}$$
Since $[a_1]_s^{l_1}$ is a maximal arithmetic progression, $a_1-s$ cannot be in the signature. Note that all terms in $D(s)$ besides $a_1-s$ are in the signature, and $\deg(s)= n-2$. This contradicts every vertex of $G_n$ has degree $n-3$. \\

\item Case: 3 arithmetic progressions.
 Let the three progressions be $[1]_s^{l_0}$, $[a_1]_s^{l_1}$, and $[a_2]_s^{l_2}$, where we assume without loss of generality that $s < a_1 < a_2$. Note that $s$ is not adjacent to either $a_1$ or $a_2$ since $[a_1]_s^{l_1}$ and $[a_2]_s^{l_2}$ are maximal progressions. By Observation~\ref{obs:missing_from_[0]}, at most 2 members are not present in the set $\{s, 2s, ..., (l-1)s\}$. \\

We will first prove that $a_1>(l_0+1)s$. Since $a_1$ is the first term of an arithmetic progression, $a_1\neq (l_0+1)s$. Assume by contradiction that $a_1<(l_0+1)s$. Then, $ps<a_1<(p+1)s$, where $1 \leq p\leq l_0s$. Note that the differences of $a_1$ with all elements in $[1]_s^{l_0}$ are less than $l_0s$. Since $a_1$ is not a multiple of $s$, the differences of $a_1$ with all elements in $[1]_s^{l_0}$ are not multiples of $s$ and not in $S$. That $a_1$ is not adjacent to any vertex in $[1]_s^{l_0}$ contradicts that $\deg(a_1)=n-3$ unless $l_0=2$. However, if $l_0=2$ and $a_1<(l_0+1)s$, then $a_1$ is not adjacent to $s$ and $2s$. Since $a_1$ is adjacent to every element of $[a_1]_s^{l_1}$, $l_1=2$. Then $l_2\geq 4$. Since $a_1$ is adjacent to every element of $[a_2]_s^{l_2}$, $S$ contains the arithmetic progression of length $l_2\geq 4$: $a_2-a_1, a_2+s-a_1,...,a_2+(l_2-1)s-a_1$, contradicting that $l_0=l_1=2$. Therefore, $a_1>(l_0+1)s$. By similar arguments, $a_2>a_1+l_1s$. \\

Consider the $i$th term of the arithmetic progression $[1]^{l_0}_s$ $is$ for $1 \leq i \leq \min \{l_0,l_1,l_2 \}$. Since $a_1-s$ and $a_2-s$ are not in the signature and $\deg(is)=n-3$, $is$ is adjacent to every vertex other than $a_1 + (i-1)s$ and $a_2 + (i-1)s$. In particular, if $2s \in S$, then $2s$ is adjacent to $a_1$ and $a_2$, and thus $|a_1-2s|$ and $|a_2-2s|$ are in $S$. 
\begin{enumerate}
\item Case: $2s\in S$. Since $a_1>(l_0+1)s$, $|a_1-2s|=a_1-2s$. 
Since $a_1-2s<a_1$ and $|a_1-2s|\in S$, then $a_1-2s\in [1]_s^{l_0}$. Let $a_1-2s=ks$, where $1\leq k\leq l_0$. Since $a_1$ is the first term of an arithmetic progression, $a_1=(k+2)s>(l_0+1)s\Rightarrow k>l_0-1$. Thus, $k=l_0$ and $a_1=(l_0+2)s$. The arithmetic progression $[a_1]_s^{l_1}$ can be written as $\{(l_0+2)s, (l_0+3)s,..., (l_0+l_1+1)s\}$.\\

By a similar reasoning, $|a_2-2s|=a_2-2s<a_2$ is in $S$ and thus, is a multiple of $s$. Let $a_2-2s=js$, where $1\leq j\leq l_0$ or $l_0+2\leq j\leq l_0+l_1+1$. Note that $a_2=(j+2)s>(l_1+l_0+2)s\Rightarrow j>l_0+l_1$. Thus, $j=l_1+l_0+1$ and $a_2=(l_0+l_1+3)s$. Therefore, 
\nonumber
\begin{align*}
S &=\{s, 2s, ..., l_0s\} \\
	&\qquad \cup \{(l_0+2)s,(l_0+3)s,\ldots, (l_0+l_1+1)s\} \\
	&\qquad \cup \{(l_0+l_1+3)s,(l_0+l_1+4)s,...,(l_0+l_1+l_2+2)s \}
\end{align*}
Note that $S$ is an arithmetic progression with two deletions. Lemma \ref{lemma:K_n-delete-2} shows that this is not a signature for $G_n$.

\item Case: $2s \notin S$. The three arithmetic progressions that partition $S$ are:
    $$
		S = \{s\} \cup [a_1]_s^{l_1} \cup [a_2]_s^{l_2}
	$$
Without loss of generality, we assume that $l_1>l_2$, and so $l_1\geq 4$ and $a_1+2s, a_1+3s\in S$. Since $2s, a_1-s\notin S$, $a_1$ is not adjacent to $s$ and $a_1+2s$. Then, $a_1$ must be adjacent to $a_1+3s$, and so $3s\in S$. Thus,
	$$
		S = \{s\} \cup \{3s,4s,\ldots,(l_1 + 2)s \} \cup [a_2]_s^{l_2}
	$$
	We know $3s$ is not adjacent to $s$ and $5s$, and so $3s$ must be adjacent to $a_2$. By assumption, $a_2>3s$. 
	We suppose $a_2-3s=qs$, where $3\leq q\leq l_1+2$ or $q=1$. Since $a_2$ is the first term of an arithmetic progression, $a_2=(q+3)s>(l_1+3)s\Rightarrow q>l_1$. This implies that $q$ can only be $l_1+1,l_1+2$. If $q=l_1+1$, then $S$ is an arithmetic progression with two deletions of $2s$ and $(l_1+3)s$. By Lemma \ref{lemma:K_n-delete-2}, this is not a signature for $G_n$. Thus, $q=l_1+2\Rightarrow a_2=(l_1+5)s$. 
	\nonumber
	\begin{align}
			S &= \{s\} \cup \{3s,4s,\ldots,(l_1 + 2)s \} \\
			&\qquad \cup \{ (l_1 + 5)s,(l_1 + 6)s, \ldots ,(n+3)s\} \\
			&= [1]_s^{n+3} \setminus \{2s, (l_1+3)s, (l_1+4)s \}
	\end{align}
	But this can never be a signature for $G_n$. If $1 \leq l_1 \leq n-4$, then $l_1+7 \leq n+3$, and thus $(l_1+7)s, (l_1+6)s$ are both in  $S$. Then, the vertex labeled $3s$ will not be adjacent to $s, (l_1+6)s$, and $(l_1+7)s$, contradicting that $\deg(3s)=n-3$. On the other hand, if $n-3 \leq l_1 \leq n-2$, then $l_1+7>n+3$ and the vertex labeled $(l_1+5)s$ is only not adjacent to $s$, contradicting that $\deg((l_1+5)s) = n-3$.
\end{enumerate}
\end{enumerate}
\end{enumerate}
We have exhausted all possible signatures for $G_n$, without finding one. This completes the proof.
\end{proof}

\section{Conclusion}
\label{sec:conclusion}

The study of autographs offers some insight into what makes certain graphs difficult or impossible to express in this particular abbreviated form. The demonstration of these specific non-autographs confirms the intuition of~\citep{bloom1984acg} that graphs whose complements have low degree are not likely to be autographs. 

\subsection{Discussion}

At first glance, one might be tempted to hope that autographs might provide a mechanism for determining if two graphs are isomorphic, since it is easy to verify that two signatures are identical. However, such a mechanism could provide a solution to the Graph Isomorphism Problem, which, while not known to be NP-complete, has no known polynomial time solution~\citep{garey2002computers}.  Thus, it is unlikely that a polynomial time algorithm for the \sig problem can be found. Indeed, the question of whether possible signature values for a graph could be bounded was raised by~\cite{bloom1979can} and addressed for certain cases by~\cite{bloom1984acg}. Thus, algorithmic searches for signatures for graphs are compromised by the facts that: i) signatures are not unique; ii) bounds for possible signature values are not known; and iii) there is no polynomial time solution for determining whether a candidate signature realizes a target graph. Nevertheless, Graph Isomorphism can be solved in polynomial time for graphs of bounded degree~\citep{luks1982isomorphism}, so there is some hope that algorithmic approaches could be viable in finding signatures for some autographs. 

\subsection{Open Problems}

Many interesting problems concerning autographs remain. Namely:

\begin{enumerate}
	\item Is the signature problem GI-complete? Such a result would clarify the computational complexity of this line of inquiry.
	\item Is it possible to derive bounds for signature values in autographs that would enable polynomial time algorithmic searches for signatures? 
	\item Are there other interesting families of non-autographs?
	\item Is there a property that determines whether a graph is an autograph?
\end{enumerate}

\subsection{Acknowledgments}


The first author is deeply indebted to Gary Bloom for introducing him to this problem and providing expert guidance towards its resolution. Prof. Bloom died unexpectedly in September of 2009 while on a hiking trip, but was reportedly in possession of an earlier version of this manuscript at the time. We dedicate this paper to his memory. 

We would also like to thank Joseph Gallian, Pavol Hell, and Valia Mitsou for helpful comments.


\bibliography{references}
\bibliographystyle{plain}

\end{document}